\documentclass[10pt]{article}

\usepackage{ifpdf}	%% if pdflatex is used
\usepackage[dutch,english]{babel}
\usepackage{type1cm}	%% use scalable, PostScript Type 1 versions of the Computer Modern fonts.

\usepackage{xspace}		% for spacing after commands
\usepackage[usenames]{color}		% use colors in the document
\usepackage{amsmath}
\usepackage{amssymb}
\usepackage{amsthm}
\usepackage{xypic}
\usepackage{url}
\usepackage{rotating}
\usepackage{array}
\usepackage{ifthen}

\ifpdf
	\usepackage[pdftex,plainpages=false,pdfpagelabels]{hyperref}
\else
	\usepackage[hypertex]{hyperref}
\fi

\newcommand{\blockpar}{\par}	% paragraph after definition, lemma, theorem, etc.
\newenvironment{xydiagram}{\begin{center}\begin{tabular}{c}}{\end{tabular}\end{center}}

\theoremstyle{plain}
\newtheorem{theorem}{Theorem}[section]
\newtheorem{lemma}[theorem]{Lemma}

\theoremstyle{definition}
\newtheorem{definition}[theorem]{Definition}
\theoremstyle{remark}
\newtheorem{remark}[theorem]{Remark}
\newtheorem{example}[theorem]{Example}

\usepackage{fancyhdr}
\usepackage{comment}

\excludecomment{acomment}

\newcommand{\afootnote}[1]{\footnote{\color{red} #1}}
\newcommand{\fracflat}[2]{{#1}/{#2}}	% inline fraction

\newcommand{\pfracTwoflat}[3]{{\partial^2{#1}}/{\partial{#2}\partial{#3}}}

\newcommand{\defemph}[1]{{\em{#1}}}	% definition, first occurrence 
\newcommand{\explemph}[1]{{\em{#1}}}	% explanation emphasis

\newcommand{\definitionref}[1]{Definition~\ref{#1}}
\newcommand{\theoremref}[1]{Theorem~\ref{#1}}
\newcommand{\lemmaref}[1]{Lemma~\ref{#1}}

\newcommand{\remarkref}[1]{Remark~\ref{#1}}

\newcommand{\sectionref}[1]{Section~\ref{#1}}
\newcommand{\MongeAmpere}{Monge-Amp\`ere\xspace}

\newcommand{\cpage}{p.\ }		% cite page
\newcommand{\cpages}{pp.\ }	% cite pages

\newcommand{\ie}{{i.e.,}\xspace}

\newif\ifsumconv	% use summation convention or not
\sumconvfalse
\ifsumconv
	\newcommand{\sumConv}[1]{}	% summation convention
	\includecomment{sumconvc}
\else
	\excludecomment{sumconvc}
	\newcommand{\sumConv}[1]{#1}	% no summation convention
\fi

\newcommand{\eskip}{\phantom{=\quad}}	% skip space
\newcommand{\composition}{\circ}	% composition
\newcommand{\eqspace}{\quad}	% basic space in equation
\DeclareMathOperator{\ev}{ev}		% evaluation map
\newcommand{\R}{\ensuremath{\mathbb{R}}\xspace}	% real numbers
\newcommand{\C}{\ensuremath{\mathbb{C}}\xspace}	 % complex numbers

\newcommand{\mathpara}{\S}	% mathematial section

\newcommand{\proj}[1]{\tilde{#1}}	% projected structures

\newcommand{\ti}[1]{\tilde{#1}}	% tilde 
\newcommand{\trans}[1]{\tilde{#1}}	%% transformed manifold, variables, bundle etc.

\newcommand{\vvM}[1]{\mathcal{X}(#1)}	% vector fields on a manifold
\newcommand{\vbdirsum}{\oplus}		% vector bundle direct sum (for subbundles)
\newcommand{\vbdual}[1]{{#1}^\perp}	% dual of distribution
\newcommand{\idealdualt}[1]{{#1}^\perp}	% dual of ideal
\newcommand{\idealduald}[1]{\left({#1}\right)^\perp}	% dual of ideal
\newcommand{\vbcompl}[1]{\textrm{compl}(#1)}

\newcommand{\E}{\mathcal{V}}	% vector subbundle/distribution
\newcommand{\V}{\mathcal{V}}	% vector subbundle/distribution, related to characteristic distributions
\newcommand{\VinD}{\subset}	% vector field contained in a distribution

\newcommand{\W}{\mathcal{W}}	% vector subbundles

\newcommand{\vbZ}{\mathcal{Z}}	% vector subbundle

\DeclareMathOperator{\vbSpan}{span}
\newcommand{\vbspand}[1]{\vbSpan\left({#1}\right)}	% vector bundle spanned by vector fields, display form
\newcommand{\vbspant}[1]{\vbSpan({#1})}	% vector bundle spanned by vector fields, text form

\newcommand{\bundleideald}[1]{\vbSpan\left({#1}\right)}	% ideal on the module of 1-forms, display form
\newcommand{\bundleidealt}[1]{\vbSpan({#1})}	% ideal on the module of 1-forms, text form

\newcommand{\I}{\mathcal{I}}	% exterior differential ideal

\newcommand{\lie}[1]{\ensuremath{\mathfrak{#1}}}	% notation for Lie algebras
\DeclareMathOperator{\GL}{GL}		% general linear group
\DeclareMathOperator{\affine}{\lie{aff}}	% affine Lie algebra

\DeclareMathOperator{\kernel}{ker}

\DeclareMathOperator{\rank}{rank}	% rank
\DeclareMathOperator{\Gr}{Gr}		% Grassmannian
\newcommand{\PP}{\mathbb{P}}	% projective spaces

\newcommand{\verz}[1]{\ensuremath{\{\,{#1}\,\}}}
\newcommand{\p}{\partial}		% partial derivative symbol
\newcommand{\ed}{\text{\upshape{d}}}		% exterior derivative operator
\newcommand{\rightinvariant}{right-in\-vari\-ant\xspace}
\newcommand{\leftinvariant}{left-in\-vari\-ant\xspace}
\newcommand{\constantrank}{constant-rank\xspace}

\newcommand{\finitedimensional}{finite-dimensional\xspace}

\newcommand{\hdecomposable}{de\-com\-po\-sable\xspace}

\ifpdf
\fi

\title{Tangential symmetries of Darboux integrable systems}

\author{P. T. Eendebak \\ Department of Mathematics \\
 Utrecht University }
\begin{document}

\maketitle

\begin{abstract}

In this paper we analyze the tangential symmetries of Darboux integrable
\hdecomposable exterior differential systems.
Our definition of a \hdecomposable exterior differential system
is a generalization of the hyperbolic systems defined in~\cite[\S 1.1]{Bryant1995-I}.
This generalization of hyperbolic exterior differential systems
 includes the classic notion of Darboux integrability for
first order systems and second order scalar equations.
For Darboux integrable systems the general solution can be found by
  integration (solving ordinary differential equations). We show that this property holds for our generalized systems as well.

We give a geometric construction of the Lie algebras of tangential symmetries associated to
the Darboux integrable systems. This construction has the advantage over previous
constructions~\cite{Vessiot1939,Vessiot1942b}, \cite{Vassiliou2001-firstorder, Vassiliou2001} that our construction
does not require the use of adapted coordinates
 and works for arbitrary dimension of the underlying manifold. In particular it works for the prolongations of \hdecomposable exterior differential systems.

Independently of the author, Anderson, Fels and Vassiliou~\cite{AFV} have developed a theory similar to the one
described in this article. Their method has a more algebraic flavor and uses adapted coframes.
Their approach allows the use of symbolic software to calculate numerous examples of Darboux integrable systems.

% An excellent background for this article is \cite[\mathpara 1.1--1.4]{Bryant1995-I}.
\end{abstract}

%\vspace{1cm} \tableofcontents

\newcommand{\tsmaniM}{M}	% equation manifold with a Darboux integrable HEDS on it
\newcommand{\tsPointM}{m}	% point in M
\newcommand{\tsPointMx}{x}	% another point in M

\newcommand{\F}{\mathcal{F}}
\newcommand{\G}{\mathcal{G}}

\newcommand{\tsHEDS}{\mathcal{M}}	% decomposable EDS
\newcommand{\tsHEDSclass}[3]{(#1,#2,#3)}	% class of decomposable EDS

\newcommand{\tsrankF}{k}	% rank of F and G
\newcommand{\tsrankG}{l}

\newcommand{\niF}{n_{\F}}	% number of invariants of \F
\newcommand{\niG}{n_{\G}}	% number of invariants of \G
\newcommand{\tsInvF}{I}
\newcommand{\tsInvG}{J}
\newcommand{\DarbouxPr}{\pi}	% Darboux projection
\newcommand{\tsdimFiber}{s}			% dimension of the fiber
\newcommand{\tsdimMani}{n}	% dimension of total manifold
\newcommand{\tsIntElem}{E}	% integral element

\newcommand{\tsBase}{B}		% base 
\newcommand{\tsBaseF}{B_1}	% base component defined by M/\vbcompl{G}
\newcommand{\tsBaseG}{B_2}	% base component defined by M/\vbcompl{G}
\newcommand{\tsBaseP}{b}	% base point 
\newcommand{\tsBasePF}{b_1}	% base point in B_1 
\newcommand{\tsBasePG}{b_2}	% base point in B_2

\newcommand{\tsS}{S}	% integral submanifold
\newcommand{\tsPointS}{s}	% point in integral submanifold
\newcommand{\tsSB}{U}	% submanifold to be lifted
\newcommand{\tsnInvariants}{k}

%% Lie algebras

\newcommand{\tsev}[2]{\ev(#1)_{#2}}		% evaluation map of Lie algebras

\newcommand{\tsf}{\lie{f}}	% Lie algebra generated by commuting vector fields
\newcommand{\tsg}{\lie{g}} 

\newcommand{\tsTangentV}{\lie{z}}		% tangential vector fields
\newcommand{\tsTangentVFib}[1]{\tsTangentV_{#1}}	% tangent vector fields on fiber
\newcommand{\tsTSf}{\tilde{\lie{f}}}		% tangential symmetries
\newcommand{\tsTSg}{\tilde{\lie{g}}}		% tangential symmetries

\newcommand{\tsfdF}[1]{\tsf'_{\tsFiberI{#1}}}	% derived Lie algebra on fiber
\newcommand{\tsgdF}[1]{\tsg'_{\tsFiberI{#1}}}

\newcommand{\tsFiber}[2]{{#1}_{#2}}		% fiber of manifold
\newcommand{\tsFiberI}[1]{#1}			% fiber index. either b or M_b
\newcommand{\tsResFiber}[1]{\rho_{\tsFiberI{#1}}}		% restriction map of vector fields

\newcommand{\tsTLF}{L}			% finite-dimensional tangential Lie algebra of symmetries
\newcommand{\tsTLG}{R}			

\newcommand{\tsCenter}{\lie{c}}		% center of the tangential Lie algebras
\newcommand{\tsCenterF}[1]{\tsCenter_{\tsFiberI{#1}}}	% center of Lie algebras on fiber
\newcommand{\tsCenterDistr}{\mathcal{C}}	% distribution spanned by center of the tangential Lie algebras

\newcommand{\tsSCmap}{\ti{\mu}}	% map from M to the structure constants of L at x
\newcommand{\tsSCtrmap}{{\mu}} % element of \GL(m) that transformation the structure coefficients
% note that \tsSCtrmap can be seen as a lift of \tsSCmap

\newcommand{\glC}{\tsSCtrmap}	% coefficient in general linear group, similar to \tsSCtrmap

%% summation over differential forms in a decomposable EDS
\newcommand{\sumAe}{i}	% first characteristic
\newcommand{\sumAt}{j}
\newcommand{\sumBe}{\alpha}	% second characteristic
\newcommand{\sumBt}{\beta}
\newcommand{\sumTe}{t}		% fiber

%%%%%% make commands blue, used for testing %%%%%%
%\input{bluecommands}

%%%%%%%%%%%%%%%%%%%%%%%%%%%%%%%%%%%%%%%%%%%%%%%%%%%%%%%%%%%%%%%%%%%%%%%%%%%%%
%%%%%%%%%%%%%%%%%%%%%%%%%%%%%%%%%%%%%%%%%%%%%%%%%%%%%%%%%%%%%%%%%%%%%%%%%%%%%

\section*{Introduction}

Darboux integrability of second order scalar equations was introduced by
Gaston Darboux~\cite{Darboux1870,Darboux1894I}.
Generalizations to many other systems are known, such as first
order systems \cite{Vassiliou2001}, hyperbolic exterior differential systems \cite{Bryant1995-I}
and elliptic equations~\cite{McKay1999}.

An important step in the classification of Darboux integrable equations was given by
Vessiot in the papers~\cite{Vessiot1939,Vessiot1942b}. Vessiot associates to
each second order equation Darboux integrable at order 2
 a Lie algebra of dimension 3 or less. He then uses the classification of
3-dimensional Lie algebras and some further computations to give a complete
classification of the Darboux integrable equations up to contact transformations.
Vassiliou~\cite{Vassiliou2001-firstorder, Vassiliou2001} gives a more geometric interpretation of the Lie algebras, which he calls
Lie algebras of tangential symmetries. The Lie algebras of Vessiot and Vassiliou are constructed using calculations in local coordinates.

In this article we introduce the concept of \hdecomposable exterior differential systems.
 The decomposable systems include all classical Darboux integrable equations as well as 
hyperbolic exterior differential systems.
 We give a geometric construction of the Lie algebras of tangential symmetries.
This geometric construction works for all dimensions of the systems. This is in
contrast with the theory developed by Vessiot and Vassiliou who only give their results
for specific dimensions.

The geometric construction has some limitations as well. The constructions requires
solving ordinary differential equations, which makes it difficult to perform
the construction in concrete examples. 
For small dimensions Vessiot used the classification of \finitedimensional Lie algebras
to classify the Darboux integrable equations.
There are two reasons why this approach fails for more our more general systems, but even
for equations that are Darboux integrable at higher order. The first reason
is that there is no tractable classification of general Lie algebras of high dimension.
The second reason is that there is no explicit construction of a Darboux integrable 
system from a given \finitedimensional Lie algebra.

In this paper we assume that, unless stated otherwise, all objects are smooth and (if applicable) of constant rank. If it is necessary to restrict
to open neighborhoods to make our constructions, then we will not
always mention this explicitly. 
\begin{sumconvc}
 In the paper we use the summation convention.
\end{sumconvc}
Part of the results in this paper have been published already, with proofs and more examples,
in~\cite[Chapter 10]{Eendebak2006}.

%%%%%%%%%%%%%%%%%%%%%%%%%%%%%%%%%%%%%%%%%%%%%%%%%%%%%%%%%%%%%%%%%%%%%%%%%

\section{Decomposable exterior differential systems}

In this section we give the definition of Darboux integrability for \hdecomposable exterior differential systems.
In \sectionref{section:distributions} we give a brief review of distributions,
in \sectionref{section:Darboux integrable decomposable systems} we give the main
definition of Darboux integrable \hdecomposable systems and finally in section 
\sectionref{subsection:other systems} we compare our definition to the
classic definitions and give some examples.

\subsection{Distributions}
\label{section:distributions}

The objects dual to (\constantrank) Pfaffian systems are distributions.
In this paper we will assume, unless stated otherwise, that all distributions are smooth and locally of constant rank.

\begin{definition}
A \defemph{distribution}\index{distribution} on a smooth manifold $\tsmaniM$ of rank $k$ is a subbundle of
the tangent bundle $T\tsmaniM$ of rank $k$.
\end{definition}
\blockpar
The distribution spanned by the vector fields $X_1, \ldots, X_n$ is denoted by $\vbspant{X_1,\ldots, X_n}$.\index{$span(X1...)$@$\vbspant{X^1,\ldots, X^n}$}

For a distribution $\V$ and a vector field $X$ we say that $X$ is
 \defemph{contained in $\V$ pointwise}
(or just \defemph{contained in $\V$})
 and write $X \VinD \V$ if $X_\tsPointM \in \V_\tsPointM$ for all points $\tsPointM$.
If $X$ is not contained in $\V$ this means that there exists a point $\tsPointM$
such that $X_m \not\in \V_\tsPointM$. This does not imply that
$X_\tsPointM \not\in \V_\tsPointM$ for all points $\tsPointM$. We will say that
$X$ is \defemph{pointwise not contained} in $\V$  if the stronger
statement, that $X_\tsPointM \not\in \V_\tsPointM$ for all $\tsPointM$, holds.

%%%

\paragraph{Invariants}

An \defemph{invariant}
for a distribution $\E$ is 
a function $I$ on $M$ such that $X(I)=0$ for all $X \subset \E$.
This is equivalent to $\E \subset \kernel(\ed I)$.
Classically, the invariants of a distribution are called
\defemph{first integrals}\index{first integral} ~\cite[\cpage 95]{Gardner1993}, \cite[\cpages 10, 289]{Stormark2000}.
We say that $\tsnInvariants$ invariants $I^1, \ldots, I^\tsnInvariants$
 are \defemph{functionally independent}\index{functionally independent}
at a point $x$ if the rank of the Pfaffian system $\bundleidealt{\ed I^1,\ldots, \ed I^\tsnInvariants}$ is equal to $\tsnInvariants$ at $x$.
By the Frobenius theorem % reference?
an integrable rank $k$ distribution on an $n$-dimensional
manifold has locally precisely $n-k$ functionally independent invariants.

\paragraph{Projections and lifting}

Let $\phi: \tsmaniM \to B$ be a smooth map. If $\phi$ is a diffeomorphism, then
we can define the push forward $\phi_* X$ of a vector field $X$
at $y=\phi(x)$ as $(\phi_* X)_y=(T_x \phi)X_x$. Locally we can
define the push forward of a vector field under an immersion in the same way.
If $\phi$ is a smooth map, then in general there is no push forward of
a vector field $X$.
 If for all points $x^1, x^2$ with $\phi(x^1)=\phi(x^2)$ the
vectors $(T_{x^1} \phi)(X)$ and $(T_{x^2} \phi)(X)$ are equal, we say that $X$ \defemph{projects down} to $B$ and we write
$\phi_* X$ for the projected vector field.
In a similar way, we can project distributions $\V$ on $M$ to $B$
if for all points $x$ in the fiber $\phi^{-1}(y)$ the image
$(T_x \phi)(\V)$ is equal to a fixed linear subspace $\W_{\phi(x)}$ of
$T_{\phi(x)}B$.

\begin{example}
Let $\phi:\R^2 = \R \times \R \to \R$ be the projection onto the first component. On
$\R^2$ take coordinates $x, y$ and define the vector fields 
\begin{align*}
X &= x \p_x , \quad
Y = x\p_x + y\p_y , \quad
Z = (1+y^2)\p_x .
\end{align*}
Under the map $\phi$ the vector fields $X$ and $Y$ project to the base manifold, the vector field $Z$
does not project. The bundle $\vbZ=\vbspant{Z}$ does project to $\R$.
\end{example}

\newcommand{\vecV}{X}\newcommand{\vecW}{Y}

\begin{lemma}[Lie brackets of projected vector fields]
\label{lemma:Lie brackets and projection}
Let $\pi:\tsmaniM \to B$ be a smooth map. Let $\vecV$, $\vecW$ be two vector fields on $\tsmaniM$
that project to vector fields $\proj{\vecV}=\pi_* \vecV$ and $\proj{\vecW}=\pi_* \vecW$ on $B$, respectively.
Then the commutator $[\vecV,\vecW]$ projects down to $B$ and 
$\pi_*[\vecV,\vecW]=[\proj{\vecV},\proj{\vecW}]$.
\end{lemma}

\subsection[Darboux integrable decomposable systems]{Darboux integrable decomposable exterior differential systems}
\label{section:Darboux integrable decomposable systems}

\begin{definition}%[Decomposable EDS]
\label{definition:decomposable EDS:ideal}
Let $\tsmaniM$ be a manifold of dimension $\tsdimMani=\tsdimFiber+\tsrankF+\tsrankG$
with an exterior differential
system $\I$ such that locally there exists a coframing
of $\tsmaniM$ of the form $\theta^1,\ldots,\theta^\tsdimFiber$,
$\omega^1,\ldots,\omega^\tsrankF$, $\omega^{\tsrankF+1},\ldots,\omega^{\tsrankF+\tsrankG}$ such
that $\I$ is generated algebraically by the forms
\begin{subequations}
\begin{align}
&\theta^1,\ldots,\theta^\tsdimFiber, \\
&\omega^\sumAe \wedge \omega^\sumAt, \quad 1 \leq \sumAe,\sumAt \leq \tsrankF, \\
&\omega^\sumBe \wedge \omega^\sumBt, \quad \tsrankF+1 \leq \sumBe,\sumBt \leq \tsrankF+\tsrankG .
\end{align}
\label{equations:structure equations HEDS}
\end{subequations}
We call $\tsHEDS=(\tsmaniM,\I)$ a \defemph{\hdecomposable exterior differential system} or a \defemph{\hdecomposable system}.
We define $I=\bundleidealt{\theta^1,\ldots,\theta^\tsdimFiber}$ and
$\V = \idealdualt{I}$.
A coframing satisfying the structure equations~\eqref{equations:structure equations HEDS}
is called an \defemph{admissible local coframing}.
\end{definition}
Note that for $\tsrankF=\tsrankG=2$ this definition corresponds to
the definition in~\cite[p. 29]{Bryant1995-I}.
We say the system has \defemph{(extended) class} $\tsHEDSclass{\tsdimFiber}{\tsrankF}{\tsrankG}$.
The systems in~\cite{Bryant1995-I} are
all of class $\tsHEDSclass{\tsdimFiber}{2}{2}$ for $\tsdimFiber \geq 0$.

\begin{acomment}
 The condition that $\I$ is closed implies that the following
structure equations hold
\begin{align*}
 \ed \theta^\sumTe &\equiv A^\sumTe_{\sumAe\sumAt} \omega^\sumAe \wedge \omega^\sumAt
		+ B^\sumTe_{\sumBe\sumBt} \omega^\sumBe \wedge \omega^\sumBt
	\mod \theta^1,\ldots,\theta^\tsdimFiber
\end{align*}
\end{acomment}

Every \hdecomposable system invariantly defines two distributions
%of rank $\tsrankF$ and rank $\tsrankG$
\begin{align}
\begin{split}
\F &= \idealduald{\bundleideald{\theta^1,\ldots,\theta^\tsdimFiber,\omega^{\tsrankF+1},\ldots,\omega^{\tsrankF+\tsrankG}}}, \\ 
\G &=\idealduald{\bundleideald{\theta^1,\ldots,\theta^\tsdimFiber,\omega^{1},\ldots,\omega^{\tsrankF}}} .
\end{split}
\label{equation:decomposable EDS:ideal to distribution}
\end{align}
This motivates the following alternative definition of a \hdecomposable exterior
differential system.
\afootnote{Anderson calls $\F^\perp$ and $\G^\perp$ the \defemph{singular Pfaffian systems}.}
%%%
%%%
\begin{definition}%[Hyperbolic EDS]
\label{definition:decomposable EDS:distribution}
Let $\tsmaniM$ be an $\tsdimMani$-dimensional manifold with two distributions $\F, \G$.
We call the triple $\tsHEDS=(\tsmaniM,\F, \G)$
\defemph{\hdecomposable system} on $\tsmaniM$ if $\F, \G$ are of constant rank and the following 
conditions hold:
\begin{enumerate}
\item The intersection of $\F$ and $\G$ is empty.
\label{enum:empty intersection}
\item $[\F,\G] \equiv 0 \mod \F \vbdirsum \G$.
\label{enum:condition closed}
\end{enumerate}
\end{definition}
\blockpar
It is not difficult to check that with the correspondence
\eqref{equation:decomposable EDS:ideal to distribution} the two
definitions~\ref{definition:decomposable EDS:ideal}
and \ref{definition:decomposable EDS:distribution} are locally equivalent.
Globally the definitions are equivalent if we allow to interchange
the two distributions $\F$ and $\G$.
The last condition in \definitionref{definition:decomposable EDS:distribution} above corresponds
to the corresponding exterior differential system being closed.
\begin{acomment}
The theory to be developed is formulated more easily in terms of distributions
and hence we will work from now on with \definitionref{definition:decomposable EDS:distribution}.
\end{acomment}

The distributions $\F,\G$ are called
the \defemph{characteristic systems} of the \hdecomposable system.
We denote by $\niF$ and $\niG$ the number of invariants of the bundles $\F$ and $\G$,
respectively. We define $\V=\F \vbdirsum \G$. In terms of exterior differential systems
$\V=\vbdual{I}$.
If the distribution $\V$ has invariants, then the completion of $\V$ is an integrable
distribution and we can restrict our structures to the leaves of $\vbcompl{\V}$.
The leaves of the integrable
distribution $\vbcompl{\V}$ each carry the structure
of a \hdecomposable system. For this reason we will assume
from here on that $\V$ has no invariants.
\begin{acomment}
  In terms of the other definition this
condition corresponds to the condition that
the infinite derived system $I^{(\infty)}$ is trivial.
\end{acomment}

The Darboux integrability of a \hdecomposable system is determined by the invariants
of the characteristic systems.
\begin{definition}%[Darboux integrable \hdecomposable system]
\label{definition:Darboux integrable decomposable system}
Let $(\tsmaniM,\F,\G)$ be a \hdecomposable system without invariants. The
system is \defemph{Darboux integrable} if $\F$ has
at least $\rank(\G)$ functionally independent invariants $\tsInvF^1,\ldots \tsInvF^{\rank(\G)}$
and
 $\G$ has at least $\rank(\F)$ functionally independent invariants $\tsInvG^1,\ldots \tsInvG^{\rank(\F)}$ such that
$\bundleidealt{\ed \tsInvF^1,\ldots, \ed \tsInvF^{\rank(\G)}, \ed \tsInvG^1, \ldots, \ed\tsInvG^{\rank(\F)}} \cap I=0.$
\end{definition}
\blockpar
\begin{acomment}
 The last condition ensures that the Darboux projection $\DarbouxPr$ is transversal to $\V$.

\end{acomment}
In the case that the \hdecomposable system is given by the contact
structure of a system of partial differential equations, then the Darboux integrability
corresponds to the classical notion of Darboux integrability.
One of the main properties of Darboux integrable equations,
 namely the construction of integral
 manifolds by integration (see \sectionref{subsection:lifting solutions}),
is also present for our Darboux integrable hyperboloc systems.
The definition above corresponds to the definition of
Darboux integrability in~\cite{Bryant1995-I} in the case
that $\rank(\F)=\rank(\G)=2$.
In \sectionref{section:Lie algebras of tangential symmetries} we will see that the Darboux integrability property leads
to a very rigid structure on the manifold.

\begin{example}
Let $\tsmaniM=\R^2$ with coordinates $x,y$. Then we can take $\F$ to be
spanned by the vector field $\p_x$ and $\G$ spanned by $\p_y$.
The bundle $\F$ has $y$ as an invariant, the bundle $\G$ has $x$ as an invariant.
The triple $(\tsmaniM,\F,\G)$ is a Darboux integrable \hdecomposable system.

In fact for every direct product $\tsmaniM_1 \times \tsmaniM_2$ the distributions $\F=T\tsmaniM_1$ and $\G=T\tsmaniM_2$ define
a Darboux integrable \hdecomposable system on $\tsmaniM_1 \times \tsmaniM_2$.
\end{example}

In this paper we will require the slightly more restrictive condition that
the number of invariants for each of the bundles $\F,\G$ is \explemph{equal} to the rank of
the other bundle. This together with the assumption of no invariants
leads to the following 2 conditions.
\begin{subequations}
\begin{align}
\dim  \vbcompl{\V} =\dim \tsmaniM \quad \text{(no invariants)}, 
\label{equation:condition no invariants}
\end{align}
\begin{align}
\niF=\rank(\G) \quad \text{ and } \quad \niG=\rank(\F) .
\label{equation:condition rank-invariants}
\end{align}
\label{decomposable EDS:assumptions}
\end{subequations}
In the case that $\niF > \rank(\G)$ or $\niG>\rank(\F)$ we can still
carry out the constructions described below, but with some
technical modifications.

Given a Darboux integrable system $\F$, $\G$ there is a natural projection onto the space of invariants. The completions of $\F$ and $\G$ are
integrable and hence they define a foliation of $M$ of codimension $\niF$ 
and $\niG$, respectively.
\begin{definition}%[Darboux projection
\label{definition:Darboux projection}
Let $(\tsmaniM,\F,\G)$ be a \hdecomposable system
satisfying \eqref{decomposable EDS:assumptions}.
Locally define $\tsBaseF$ to be the quotient of $\tsmaniM$ by the completion of $\G$ and
$\tsBaseG$ to be the quotient of $\tsmaniM$ by the completion of $\F$. Let $\DarbouxPr_1$ and $\DarbouxPr_2$ be
the projection of $\tsmaniM$ on $\tsBaseF$ and $\tsBaseG$, respectively.
The projection $\DarbouxPr = \DarbouxPr_1 \times \DarbouxPr_2: \tsmaniM \to \tsBase = \tsBaseF \times \tsBaseG$ is a natural projection
onto the manifold $\tsBase=\tsBaseF \times \tsBaseG$ of dimension $\niG + \niF$.
We call such a projection a \defemph{Darboux projection}\index{Darboux projection}.
\end{definition}
\blockpar
\begin{acomment}
 The Darboux projection is a projection onto the space of invariants. If $\niF > \rank \G$
or if $\niG > \rank \F$, then we have to project to a \explemph{subspace} of the space
of invariants such that the projection is transversal to the distribution $\V$.
In this paper we will not discuss this possibility.
\end{acomment}
The tangent spaces to the fibers of the Darboux projection are equal to the integrable distribution
$\vbZ = \vbcompl{\F} \cap \vbcompl{G} = \ker \DarbouxPr$.
We write $\tsTangentV$ 
 for
the Lie algebra of vector fields tangent to the projection.
The vector fields in $\tsTangentV$ are precisely the
vector fields in the distribution
$\vbZ$.
The invariants of $\F$ and $\G$ are precisely the functions in
 $\DarbouxPr_1^*(C^\infty(\tsBaseF))$ and $\DarbouxPr_2^*(C^\infty(\tsBaseG))$, respectively.
%The functions that are invariants of $\G$ are functions in
% $\DarbouxPr_2^*(C^\infty(\tsBaseG))$.

There is a natural isomorphism of $T(\tsBaseF \times \tsBaseG)$
with $T\tsBaseF \otimes T\tsBaseG$. Using this identification we have the following lemma.
\begin{lemma}
\label{lemma:F and G project}
The distributions $\F$ and $\G$ project to $\tsBase$. The image of $\F$ is equal 
to $T\tsBaseF \times \verz{0}$ and the image of $\G$ is equal to $ \verz{0} \times T\tsBaseG$.
\end{lemma}
\begin{proof}
The condition in \eqref{equation:condition rank-invariants}
 implies that $\dim \tsBase=\rank \V$. The together
with the fact that $\V$ has no invariants implies that the projection of 
$\V$ is onto $T\tsBase$.
Since $\F$ is contained in $\vbcompl{\F}$ and $\tsBaseG$ is defined locally as the foliation of $\tsmaniM$ 
by the leaves of the completion of $\F$, the projection of %vectors in
 $\F$ is contained
in $T\tsBaseF \times \verz{0}$. Since $\F$ has rank $\niG$ and
$\F$ is transversal to the projection $\DarbouxPr$ (since $\V$ is transversal), the
image of $\F$ under $T_m\DarbouxPr$ has rank $\niG$ and must be equal
to $T\tsBaseF \times \verz{0}$.
For $\G$ a similar argument works.
\end{proof}
\blockpar

%%%%%%%%%%%%%%%%%%%%%%%%%%%%%%%%%%%%%%%%%%%%%%%%%%%%%%%%%%%%%%%%%%%%%%%%%
%%%%%%%%%%%%%%%%%%%%%%%%%%%%%%%%%%%%%%%%%%%%%%%%%%%%%%%%%%%%%%%%%%%%%%%%%

Since the bundles $\F$ and $\G$ project nicely onto $\tsBase$, we can lift vectors and
vector fields on $\tsBase$ to vectors and vector fields on $\tsmaniM$. Another way of saying this
is that $\V=\F \vbdirsum \G$ provides a connection
for the bundle $\tsmaniM \to \tsBase$.
We have
\begin{align*}
T\tsmaniM &= \vbcompl{\F} \vbdirsum \G = \F \vbdirsum \vbcompl{\G} \\
	&= \F \vbdirsum (\vbcompl{\F} \cap \vbcompl{\G}) \vbdirsum \G \\
	&= \F \vbdirsum \vbZ \vbdirsum \G .
\end{align*}
\begin{acomment}
 The content of the lemma can be expressed with the following commutative diagram
\begin{center}
\begin{tabular}{c}
\xymatrix{
\tsmaniM \ar[d]^{\DarbouxPr} & & \V \ar[r] \ar[d]^{T\DarbouxPr} \ar@{}[dr]|-{\circlearrowleft} & \F \vbdirsum \G \ar[d]^{T\DarbouxPr_1 \times T\DarbouxPr_2} \\
\tsBase & & T\tsBase \ar[r] & T\tsBaseF \times T\tsBaseG \\
}
\end{tabular}
\label{diagram:HEDS:characteristic systems}
\end{center}
\end{acomment}

\subsection{Integral elements and prolongations}
\label{subsection:integral elements}
\label{subsection:prolongation}

\begin{definition}
\label{definition:integral element of decomposable system}
Let $\tsHEDS=(\tsmaniM,\F,\G)$ be a \hdecomposable system.
We define a 2-dimensional linear subspace $\tsIntElem$ of $\V_\tsPointM$ to be
an \defemph{integral element} of $\tsHEDS$ if
$\dim \tsIntElem\cap \F_\tsPointM =\dim \tsIntElem \cap\G_\tsPointM=1$.
\end{definition}
For a \hdecomposable system an integral element in the sense
of \definitionref{definition:integral element of decomposable system} above 
corresponds to the definition of an integral
element in the ordinary sense of the corresponding exterior differential system $\I$ \cite{BCGGG}.
The space of 2-dimensional integral elements of a \hdecomposable system has a very simple structure.
\begin{lemma}
 Let $(\tsmaniM,\F,\G)$ be a \hdecomposable system.
The map
\begin{align*}
\PP \F \vbdirsum \PP \G 
\to \Gr_2(T \tsmaniM): (f,g) \mapsto f+g
\end{align*}
defines an isomorphism from $\PP \F \vbdirsum \PP \G$
to the space of 2-dimensional integral elements of the \hdecomposable system.
\end{lemma}
\begin{acomment}
Note that in the definition of integral elements of
a \hdecomposable system we have excluded
the ``degenerate'' integral elements. In particular
the integral elements of $\V$ are not all integral elements of the \hdecomposable system. These degenerate elements are
automatically excluded in the exterior differential system formulation.
\end{acomment}

For our \hdecomposable systems we can define prolongations in a similar
way to \cite[\mathpara 1.3]{Bryant1995-I}.
Recall that a \hdecomposable system is of class $\tsHEDSclass{\tsdimFiber}{\tsrankF}{\tsrankG}$
if $\dim \tsmaniM =\tsdimFiber+\tsrankF+\tsrankG$, $\rank \F=\tsrankF$ and $\rank \G=\tsrankG$.
 The prolongation of a system of
 class $\tsHEDSclass{\tsdimFiber}{\tsrankF}{\tsrankG}$
is a \hdecomposable system of class $\tsHEDSclass{\tsdimMani+(\tsrankF-1)+(\tsrankG-1)}{\tsrankF}{\tsrankG}$.
For $\tsrankF=\tsrankG=2$ this corresponds to
\cite[Proposition, \cpage 53]{Bryant1995-I}.

\begin{acomment}
 \begin{remark}
  To prolong a \hdecomposable exterior differential system introduce
coordinates $c_\sumAe, d_\sumBe$ for the space of integral elements by
\begin{align*}
 \theta^\sumTe &= 0, \\
 \Theta^\sumAe &= \omega^\sumAe  - c_\sumAe \omega^1 = 0, \\
 \Theta^{\sumBe} &= \omega^\sumBe  - d_\sumBe \omega^1 = 0 .
\end{align*}
Then the prolonged exterior differential system on
$\tsmaniM \times \R^{\tsrankF -1 + \tsrankG -1}$ is given by
\begin{align*}
 &\theta^t, \quad 1 \leq t \leq \tsdimFiber, \\
 &\Theta^\sumAe, \quad 2 \leq \sumAe \leq \tsrankF, \\
 &\Theta^\sumBe, \quad \tsrankF+2 \leq \sumAe \leq \tsrankF+\tsrankG .
\end{align*}
together with the 2-forms formed by the pairs
\begin{align*}
&\omega^1, \  \Omega^\sumAe =  \ed c_\sumAe,  \\
&\text{and } \  \omega^{\tsrankF+1} , \  \Omega^\sumBe = \ed d_\sumBe . 
\end{align*}
Done!
 \end{remark}
\end{acomment}

\subsection{Lifting solutions}
\label{subsection:lifting solutions}

We define an integral manifold (a ``solution'') of a \hdecomposable system
$(\tsmaniM,\F,\G)$
to be a 2-dimensional submanifold $\tsS$ of $\tsmaniM$ such that
for all points $\tsPointS\in \tsS$ the tangent space $T_\tsPointS \tsS$ is an 
integral element of the system.
The integral manifolds of a \hdecomposable system are precisely the 2-dimensional integral manifolds of the bundle $\V=\F \vbdirsum \G$
that satisfy the non-degeneracy condition that at each point the tangent space of the integral manifold intersected with both $\F$ and $\G$ is non-empty.
%%%
In this section we will show that for a Darboux integrable \hdecomposable system
the integral manifolds can be found by solving \explemph{ordinary differential equations}.
In contrast, for a general \hdecomposable system the integral
manifolds are the solutions of a system of \explemph{partial differential equations}.
In the case $\tsrankF=\tsrankG=2$ this system of partial differential equations is \hdecomposable (in the sense of partial differential equations), hence the name \explemph{\hdecomposable} systems.

Assume $(\tsmaniM,\F,\G)$ is a Darboux integrable \hdecomposable system
% without invariants.
satisfying~\eqref{decomposable EDS:assumptions}.
%%%
Let $\DarbouxPr : \tsmaniM \to \tsBase$ be the Darboux projection of the system.
We can parameterize the integral manifolds of this system in the following way.
Select a curve $\gamma_1$ in $\tsBaseF$ and a curve $\gamma_2$ in $\tsBaseG$. The product
of these two curves is a surface $\tsSB$ in $B$. Let $\tsS=\DarbouxPr^{-1}(\tsSB)$ be
the inverse image of $\tsSB$. The distribution $\V$ restricts
on the inverse image $\tsS$ to an integrable distribution $\W$ of rank $2$.
\begin{acomment}
\afootnote{Include this additional explanation?}
To show that $\W$ is an integrable distribution note that
for $X,Y \VinD \W$ we have $[X,Y] \subset T\tsS$.
From the definition of $\tsS$ and \lemmaref{lemma:F and G project} it follows that
$\rank T\tsS \cap \F=\rank T\tsS \cap \G=1$.
Then from $[\F,\G] \equiv 0 \mod \V$ it follows that
for $X,Y \VinD \W$ we have $[X,Y] \VinD \V$.
Combining this we see that for $[X,Y] \VinD \W$
we have $[X,Y] \subset \V \cap T\tsS = \W$.
\end{acomment}
 The leaves of this distribution are integral manifolds of
the \hdecomposable system. % of $\V$.
Since $\W$ is integrable, finding the integral manifolds can be done
by solving ordinary differential equations. The integral manifolds
obtained in this way depend on two arbitrary functions of
$\tsrankF-1$ and $\tsrankG-1$ variables (to determine
the curves $\gamma_1$ and $\gamma_2$) and $\tsdimFiber$ integration constants.
It is not difficult to show that locally every integral manifold
of the \hdecomposable system is given in this way as the \explemph{lift} of
two curves $\gamma_1, \gamma_2$.

\subsection{Examples}
\label{subsection:examples}
\label{subsection:other systems}

Hyperbolic exterior differential systems are discussed in detail
in~\cite{Bryant1995-I, Bryant1995-II}. It is shown that the equation manifolds with the
contact structure of 
\MongeAmpere equations, first order systems and second order scalar equations 
can all be realized as 
\hdecomposable exterior differential systems of class $s=1$, 2 and 3, respectively,
see~\cite[Example 5, 6, 7]{Bryant1995-I}.
The solutions of the equations correspond to the integral manifolds of the \hdecomposable
exterior differential systems.
All Darboux integrable equations in the classes mentioned above provide examples  of Darboux integrable \hdecomposable systems.

Vassiliou defines in~\cite[Definition 2.10]{Vassiliou2001} the concept
of a manifold of $(p,q)$-\hdecomposable type.
In terms of our \definitionref{definition:decomposable EDS:distribution}
a manifold of $(p,q)$-\hdecomposable type
is a \hdecomposable exterior differential system without invariants, but with
the additional conditions that $\F$ and $\G$ are non-integrable and
\begin{align}
 \rank \F=\rank \G=2, \ \dim \tsmaniM \geq 6, \ \niF=p, \ \niG=q.
\label{equation:assumptions manifold pq-type}
\end{align}
Condition 2b) in the definition of Vassiliou corresponds to the assumption of no invariants.
Under the additional assumptions~\eqref{equation:assumptions manifold pq-type} one
can easily show that condition 2c) corresponds to condition \ref{enum:condition closed}.
The final condition 2d) defines the number of invariants of the
distributions $\F$ and $\G$. The main case in
the work of Vassiliou are the manifolds of $(2,2)$-\hdecomposable type
and these correspond to our definition of a Darboux integrable \hdecomposable system.

\begin{example}
Consider the equation manifold associated to the Liouville equation $z_{xy}=\exp(z)$.
For the equation manifold in the second order jet bundle we use the coordinates $x$, $y$, $z$, $p=z_x$, $q=z_y$, $r=z_{xx}$, $t=z_{yy}$. 
On this manifold $\tsmaniM$ we have two natural distributions defined by the characteristic
systems. In coordinates $x,y,z,p,q,r,t$ we have
\begin{align*}
\F &= \vbspand{ \p_x + p\p_z + r\p_p + \exp(z)\p_q + q\exp(z)\p_t, \p_r} , \\
\G &= \vbspand{ \p_x + q\p_z + \exp(z) \p_p + t\p_q + p\exp(z)\p_r, \p_t} .
\end{align*}
The contact structure on the equation manifold is given by
$\V=\F \vbdirsum \G$.
The bundle $\F$ has two invariants $y, t-q^2/2$ and the bundle $\G$ has
two invariants $x, r-p^2/2$. The triple $(\tsmaniM,\F,\G)$ is a Darboux integrable
\hdecomposable system.
Integration of this system leads to the general solution of the Liouville equation
\begin{align*}
 z(x,y) = \ln\left( \frac{2\phi'(x)\psi'(y)}{(\phi(x)+\psi(y))^2} \right) .
\end{align*}
Here $\phi, \psi$ are 2 arbitrary functions.
% See \cite[\cpage 20]{JurasThesis} or \cite{??}
\end{example}

Finally we note that for many elliptic systems of partial differential equations
we can define the equivalent notion of an ``elliptic exterior differential system''. For these systems we can also define the Darboux integrability
property and construct solutions using ordinary integration. 
See~\cite[\mathpara 8.1.4]{Eendebak2006} for more details.

\begin{acomment}
 
Note that to define and use elliptic Darboux integrability we only
need to complexify the tangent space and do not need to complexify the
manifold $M$. Complexifying the entire manifold $M$ is
a usefull technique used frequently in the 19th century, but this
requires $M$ and $\V$ to be analytic.
\end{acomment}

%%%%%%%%%%%%%%%%%%%%%%%%%%%%%%%%%%%%%%%%%%%%%%%%%%%%%%%%%%%%%%%%%%%%%%%%%
%%%%%%%%%%%%%%%%%%%%%%%%%%%%%%%%%%%%%%%%%%%%%%%%%%%%%%%%%%%%%%%%%%%%%%%%%

\section{Lie algebras of tangential symmetries}
\label{section:Lie algebras of tangential symmetries}

In this section we will define and construct the
tangential symmetries of Darboux integrable systems.
The constructions here are \explemph{local} and work in the smooth
setting. We will assume all \hdecomposable systems have no invariants.
We start in \sectionref{subsection:reciprocal Lie algebras} with some results on
reciprocal Lie algebras necessary later on.

\begin{definition}
\label{definition:tangential symmetries}
Let $(\tsmaniM,\F,\G)$ be a Darboux integrable \hdecomposable system
with Darboux projection $\DarbouxPr: \tsmaniM \to \tsBaseF \times \tsBaseG$.
We define the \defemph{tangential symmetries}
\index{tangential symmetry} of $\F$ and $\G$
as the space of all vector fields in $\tsTangentV$ that are symmetries of the
distributions $\F$ and $\G$, respectively.
We write $\tsTSf$ and $\tsTSg$ for the tangential
symmetries of $\G$ and $\F$, respectively.
The vector fields in $\tsTangentV$ that are symmetries of both $\F$ and $\G$
are called \defemph{tangential symmetries} of the \hdecomposable system.
\end{definition}
\blockpar
The name tangential characteristic symmetries was
introduced by Vassiliou~\cite{Vassiliou2001}.
The main results of this paper are the two theorems below.
They are both proved by the constructions in 
\sectionref{subsection:geometric construction}.

\begin{theorem}
\label{theorem:tangential symmetries}
Let $(\tsmaniM,\F,\G)$ be a Darboux integrable \hdecomposable system
with projection $\DarbouxPr: \tsmaniM \to \tsBaseF \times \tsBaseG$.
The space of tangential symmetries of $\F$ ($\G$) is a
\finitedimensional module over $\DarbouxPr_1^*(C^\infty(\tsBaseF))$ (over $\DarbouxPr_1^*(C^\infty(\tsBaseG))$) of dimension $\tsdimFiber$.
\end{theorem}
\blockpar

%  We can compare our main theorem to \cite[Theorem 4.1]{Vassiliou2001}.
\begin{theorem}[Main theorem]
\label{theorem:main theorem}
Let $(\tsmaniM,\F,\G)$ be a Darboux integrable \hdecomposable system
with projection $\DarbouxPr: \tsmaniM \to \tsBaseF \times \tsBaseG$.

There exist finite-dimensional Lie algebras of vector field $\tsTLF$, $\tsTLG$ on $\tsmaniM$ tangential to the projection $\pi$
such that $\tsTLF$ and $\tsTLG$ commute and the elements in $\tsTLF$ and $\tsTLG$ are tangential symmetries of $\G$ and
$\F$, respectively. 
\end{theorem}

\subsection{Reciprocal Lie algebras}
\label{subsection:reciprocal Lie algebras}

Let $\tsmaniM$ be a smooth manifold. We denote the space of smooth vector fields on $\tsmaniM$
by $\vvM{\tsmaniM}$.
Let $\lie{g}$ be an $n$-dimensional Lie algebra and let
$\alpha: \lie{g} \to \vvM{\tsmaniM}$ be a representation of 
$\lie{g}$ in the space of vector fields on $\tsmaniM$.
In the theory to be developed below we will work locally and $\tsmaniM$ will often be of the
same dimension as $\lie{g}$, so we can think of $\tsmaniM$ as an open subset
of $\R^n$.
We say the representation is \defemph{locally transitive} if
$\alpha(\lie{g})$ locally generates as a $C^\infty(\tsmaniM)$-module
the space of vector fields $\vvM{\tsmaniM}$.
If $\dim \tsmaniM=\dim \lie{g}=n$, then a transitive representation defines
an injective map $\lie{g} \to \vvM{\tsmaniM}$ and we can 
identify $\lie{g}$ with its representation as vector fields on $\tsmaniM$.

A Lie algebra of vector fields on $\tsmaniM$ is a Lie subalgebra of $\vvM{\tsmaniM}$.
For any Lie algebra $\lie{g}$ of vector fields on $\tsmaniM$ and a point $\tsPointM$ in $\tsmaniM$
we can define the \defemph{evaluation map}
\begin{align}
\tsev{\lie{g}}{\tsPointM} &: \lie{g} \to T_\tsPointM \tsmaniM: X \mapsto X(\tsPointM) .
\label{equation:tangential symmetries:definition:evaluation map}
\end{align}
A Lie algebra $\lie{g}$ of vector fields on $\tsmaniM$ is \defemph{locally transitive}
at $\tsPointM \in \tsmaniM$
if and only if the
evaluation map $\tsev{\lie{g}}{\tsPointM}$ is a linear
isomorphism from $\lie{g}$ onto $T_\tsPointM \tsmaniM$.
From here on we will identify a Lie algebra with its representation as vector fields whenever
the representation is injective.

%%%%%%%%%%%%%%

\begin{acomment}
We define the \defemph{centralizer}\index{centralizer} of $\lie{g}$ to
be the Lie algebra of vector fields that commute with $\lie{g}$.
\end{acomment}

The following theorem and lemma are both proved in~\cite[\cpages 242--243]{Eendebak2006}.
\begin{theorem}
\label{theorem:reciprocal Lie algebras}
Let $\tsmaniM$ be a smooth $n$-dimensional manifold. Let $\lie{g}$ be
an $n$-dimensional locally transitive Lie subalgebra of $\vvM{\tsmaniM}$.
%Let $\lie{h}$ be an $n$-dimensional locally transitive Lie algebra of
%vector fields on $\R^n$.
Then the centralizer $\lie{h}$ of $\lie{g}$
is an $n$-dimensional locally transitive Lie algebra of vector fields.
The Lie 
algebra $\lie{g}$ is anti-isomorphic with
$\lie{h}$ in the sense that for every point $x\in M$
the invertible linear map $\alpha_x=(\tsev{\lie{h}}{x})^{-1} \composition \tsev{\lie{g}}{x}$ is
a Lie algebra anti-homomorphism.
In particular for all vector fields $X,Y \in \lie{g}$ we have
$\alpha_x([X,Y]) = - [\alpha_x(X),\alpha_x(Y)]$.
\end{theorem}
We call a pair of commutative, locally transitive Lie algebras
a pair of \defemph{reciprocal Lie algebras}.
For a pair of reciprocal Lie algebras $\lie{g}$, $\lie{h}$ the center
of $\lie{g}$ is equal to the center of $\lie{h}$.

\begin{example}[Reciprocal Lie algebra]
\index{reciprocal Lie algebra!of affine Lie algebra}
Consider the affine Lie algebra $\affine(1)$, represented by the two vector fields
\begin{align*}
e_1 &= \p_{x^1} - x^2 \p_{x^2}, \quad e_2 = \p_{x^2} .
\end{align*}
Then the reciprocal Lie algebra $\lie{h}$ is generated by 
\begin{align*}
f_1 &= \p_{x^1} , \quad f_2 = \exp(-x^1) \p_{x^2} .
\end{align*}
Note that $[e_1,e_2]=e_2$ and $[f_1,f_2]=-f_2$ so the structure constants for both Lie algebras
are related by a minus sign.
The Lie algebras $\affine(1)$ and $\lie{h}$ commute in the sense that
$[e_i, f_j]=0$, $1 \leq i,j \leq 2$.
\end{example}

We will end this section with a lemma on reciprocal Lie algebras.
We will use this lemma in \sectionref{subsection:geometric construction}.
\begin{lemma}
\label{lemma:reciprocal Lie algebras:surjective implies injective}
Let $\tsmaniM$ be a smooth connected manifold with Lie algebras
of vector fields $\lie{g}$, $\lie{h}$.
We assume that i) $\lie{g}$ and $\lie{h}$ commute, \ie for all $X\in \lie{g}$ and $Y\in \lie{h}$ we have\
$[X,Y]=0$
and ii) for all $\tsPointMx\in \tsmaniM$ the evaluation
maps $\tsev{\lie{g}}{\tsPointMx}: \lie{g} \to T_\tsPointMx \tsmaniM$ and
$\tsev{\lie{h}}{x}: \lie{h} \to T_\tsPointMx \tsmaniM$ are surjective.
Then for all $\tsPointMx\in \tsmaniM$ the maps $\tsev{\lie{g}}{\tsPointMx}$
and $\tsev{\lie{h}}{\tsPointMx}$ are injective and hence
$\lie{g}$, $\lie{h}$ are reciprocal Lie algebras of
dimension equal to the dimension of $\tsmaniM$.
\end{lemma}

The Lie algebra of a Lie group $G$ is equal to the tangent space at the identity element $\lie{g}=T_e G$. The Lie algebra can be identified with both the space
of \leftinvariant vector fields on $G$ and the space of \rightinvariant vector fields
on $G$. The spaces of left- and \rightinvariant vector fields on a Lie group
are reciprocal Lie algebras \cite[\cpages 41,42]{Duistermaat2000}.
% \cite[\mathpara 10.2]{Eendebak2006}

\subsection{Geometric construction}
\label{subsection:geometric construction}

In this section we will show how to construct the tangential symmetries
of Darboux integrable \hdecomposable systems.

\paragraph{Commuting vector fields}
Let $(\tsmaniM,\F,\G)$ be a Darboux integrable \hdecomposable system
with Darboux projection $\DarbouxPr:\tsmaniM \to \tsBase$.
We assume that $\V=\F \vbdirsum \G$ has no invariants and
condition~\eqref{equation:condition rank-invariants} holds.
Select locally a basis of commuting vector fields $\ti{F}_1, \ldots, \ti{F}_{\niG}$,
 $\niG=\rank \F$ for $\tsBaseF$
and $\ti{G}_1, \ldots, \ti{G}_{\niF}$ for $\tsBaseG$, \ie $[\ti{F}_i,\ti{F}_j]=0$, $[\ti{G}_i,\ti{G}_j]=0$.
As vector fields on $\tsBase=\tsBaseF \times \tsBaseG$ we then automatically have $[\ti{F}_i, \ti{G}_j]=0$.
We can lift these vector fields to unique vector fields in $\tsmaniM$ by requiring that the lifted vector fields
are contained in $\V$. We write $F_i$ and $G_j$ for the lift of
$\ti{F}_i$ and $\ti{G}_j$, respectively.
%From this construction it follows $F_i \in \F$ and $G_i \in \G$.
Since the vector fields $F_i$ and $G_j$ are contained in $\F$ and $\G$, respectively, their Lie brackets 
$[F_i,G_j]$ must be contained in $\V$.
On the other hand, the projections have Lie bracket $[\ti{F}_i,\ti{G}_j]$
equal to zero and therefore $[F_i, G_j]$ must be contained in the tangent space of the fibers of the projection.
But $\V$ is transversal to the fibers of the projection and it follows that $[F_i, G_j]=0$.
We will use the vector fields $F_i$ and $G_j$ to construct various Lie algebras
on the fibers of the projection. We only make a \explemph{choice} of vector fields
to make the constructions and the proofs easier, most of the Lie algebras that we 
construct are \explemph{independent} of the particular choice of $F_i$ and $G_j$.

We define $\tsf$ as the Lie algebra of vector fields (over $\R$, not over $\C^\infty(\tsmaniM)$)
generated by $F_i$, $1 \leq i \leq \rank \F$. This Lie algebra is contained
in the Lie algebra of vector fields in the completion of $\F$ and is not necessarily finite-dimensional.
We define $\tsg$ as the Lie algebra of vector fields (over $\R$)
generated by $G_i$, $1 \leq i \leq \rank \G$.

For two vector fields $F_i, F_j$ the Lie bracket $[F_i,F_j]$ is tangent to
the fibers of the projection $\DarbouxPr$.
This follows from the fact that the $F_i$ are lifts of commuting vector fields and
hence $T\DarbouxPr([F_i,F_j]) = [\ti{F}_i,\ti{F}_j]=0$.
This implies 
that the derived Lie algebras $\tsf'$ and $\tsg'$ consist of vector fields
that are tangential to the projection.
Since the generators $F_i$ for $\tsf$ are not tangential,
we conclude $\tsf'=\tsf \cap \tsTangentV$.
The elements of $\tsf'$ commute with the elements in $\tsg$ and hence 
the elements of $\tsf'$ are symmetries of $\G$ that are tangential
to the projection $\DarbouxPr$.

The discussion above shows that the vector fields in $\tsf'$
are all tangential symmetries of $\G$.
 We will see
below~(also see \lemmaref{lemma:tangential symmetries:x1}) that the
tangential symmetries can be expressed in terms
of the Lie algebras $\tsf'$ and $\tsg'$.

\paragraph{Restriction to fibers}
For every point $\tsBaseP \in \tsBase=\tsBaseF \times \tsBaseG$ we 
write $\tsFiber{\tsmaniM}{\tsBaseP}$ for the fiber $\DarbouxPr^{-1}(\tsBaseP)$.
We write $\tsTangentVFib{\tsBaseP}$ for the vector fields
on $\tsFiber{\tsmaniM}{\tsBaseP}$.
For the tangential vector fields $\tsTangentV$ we define the
restriction map
\begin{align}
\tsResFiber{\tsBaseP}: \tsTangentV \to \tsTangentVFib{\tsBaseP} .
\label{equation:tangential symmetries:restriction map:general}
\end{align}
The restriction map is a Lie algebra homomorphism.

Since the vector fields in $\tsf'$ and $\tsg'$ are
tangential we can define the
restriction maps
$\tsResFiber{\tsBaseP}: \tsf' \to\tsTangentVFib{\tsBaseP}$ 
%\label{equation:tangential symmetries:restriction map:f}
and
$\tsResFiber{\tsBaseP}: \tsg' \to \tsTangentVFib{\tsBaseP}$ as well.
We denote the image of $\tsf'$ under $\tsResFiber{\tsBaseP}$ by $\tsfdF{\tsBaseP}$
and the image of $\tsg'$ under $\tsResFiber{\tsBaseP}$ by $\tsgdF{\tsBaseP}$.
Since the
Lie algebras $\tsf'$ and $\tsg'$ commute,
the Lie algebras $\tsfdF{\tsBaseP}$ and $\tsgdF{\tsBaseP}$
are commuting Lie algebras of
 vector fields on $\tsFiber{\tsmaniM}{\tsBaseP}$.

Since $\F$ has $\niF$ invariants, the codimension of the completion of $\F$
is equal to $\niF$ on an open dense subset. The same is true for $\G$.
In the constructions that
follow we will always assume that we have restricted to
such open dense subsets.
%
%From this it follows that on an open dense subset 
Under this assumption
the vector fields in the Lie algebra
$\tsf'$ span the tangent space to the fiber.
We can choose a set of vector fields $X_1, \ldots, X_{\tsdimFiber}$
($\tsdimFiber$ is the dimension of the fibers) in $\tsf'$ such
that the restriction of these vector fields to a fiber $\tsFiber{\tsmaniM}{\tsBaseP}$
is a basis for $\tsfdF{\tsBaseP}$.
Let $X$ be a tangential symmetry of $\G$. Since
the vector field $X$ is tangential 
we can write $X=\sumConv{\sum_{j=1}^\tsdimFiber} c^j X_j$ for certain functions $c^j$.
Since $X$ is a tangential symmetry, the commutator of $X$ with $\G$ is contained
in $\G$ and hence for all $Y \subset \G$
\begin{align*}
[X,Y] &\equiv \sumConv{\sum_{j=1}^\tsdimFiber} [c^j X_j, Y] \equiv \sumConv{\sum_{j=1}^\tsdimFiber} c_j [X_j, Y] -Y(c^j)X_j  \equiv \sumConv{\sum_{j=1}^\tsdimFiber} Y(c_j) X_j \equiv 0 \mod \G.
\end{align*}
This implies $Y(c_j)=0$ for all $Y \subset \G$ and hence the $c^j$
are functions of the invariants of $\G$ only, so
$c^j \in \DarbouxPr_1^*(C^\infty(\tsBaseF))$.
This proves that the tangential symmetries of $\G$ are
a $\DarbouxPr_1^*(C^\infty(\tsBaseF))$-module over
the Lie algebra $\tsf'$.
We have proved
\begin{lemma}
\label{lemma:tangential symmetries:x1}
The tangential symmetries of $\G$ are a $\DarbouxPr_1^*(C^\infty(\tsBaseF))$-module over $\tsf'$.
The tangential symmetries of $\F$ are a $\DarbouxPr_2^*(C^\infty(\tsBaseG))$-module over $\tsg'$.
\end{lemma}
\blockpar
The lemma above also proves~\theoremref{theorem:tangential symmetries}.

The Lie algebras $\tsf'$ and $\tsg'$ depend on the choice of commuting
vector fields $\ti{F}_i$ and $\ti{G}_j$, respectively. The Lie algebras
of tangential symmetries $\tsTSf$ and $\tsTSg$ are independent of this choice (this is clear from \definitionref{definition:tangential symmetries}).
In the lemma below we show that the Lie algebras $\tsfdF{\tsBaseP}$ and
$\tsgdF{\tsBaseP}$ are also independent of the choice of commuting vector fields.

\begin{lemma}
For every point $\tsBaseP \in \tsBase$ the Lie algebras
$\tsfdF{\tsBaseP}$ and $\tsgdF{\tsBaseP}$ on $\tsFiber{\tsmaniM}{\tsBaseP}$ are invariantly defined
reciprocal 
Lie algebras.
The type of the Lie algebra does not depend on the point $\tsBaseP\in \tsBase$.
\end{lemma}
\begin{proof}
The Lie algebra $\tsf'$ is tangential to the fibers of the projection. 
Since $\F$ has only $\niF=\rank \G$ invariants, it follows
that for $y$ in an open dense subset of $\tsmaniM$ the image
of the evaluation map $\tsev{\tsf'}{y}: \tsf' \to T_y \tsmaniM_{\DarbouxPr(y)}$ spans
the tangent space $T_{y} M_{\DarbouxPr(y)}$.
This in turn implies that for all points $x$ in an open subset of $\tsFiber{\tsmaniM}{\tsBaseP}$
the evaluation map $\tsev{\tsfdF{\tsBaseP}}{x}: \tsfdF{\tsBaseP} \to T_x \tsFiber{\tsmaniM}{\tsBaseP}$ is surjective.
The same is true for $\tsgdF{\tsBaseP}$.
Hence we can apply
\lemmaref{lemma:reciprocal Lie algebras:surjective implies injective}
to $\tsfdF{\tsBaseP}$ and $\tsgdF{\tsBaseP}$ and
conclude that $\tsfdF{\tsBaseP}$ and $\tsgdF{\tsBaseP}$ are reciprocal Lie algebras on
$\tsFiber{\tsmaniM}{\tsBaseP}$.
From the definitions it follows directly that the Lie algebra $\tsfdF{\tsBaseP}$ only depends on the choice of vector fields $F_i$, and not on the choice of the vector fields $G_j$.
In the same way $\tsgdF{\tsBaseP}$ does only depend on the choice of $G_j$.
On the other hand, $\tsfdF{\tsBaseP}$ is the centralizer of $\tsgdF{\tsBaseP}$ in the fiber $\tsFiber{\tsmaniM}{\tsBaseP}$
and hence $\tsfdF{\tsBaseP}$ only depends on the choice of $G_j$. This implies that
$\tsfdF{\tsBaseP}$ is invariantly defined.
For the same reason $\tsgdF{\tsBaseP}$ is invariantly defined.

Make a choice of $\tsdimFiber$ vector fields $X_i$ in $\tsf'$ such that
at each point
the vector fields span the tangent space to the fibers of the projection.
Locally, near a point $x\in \tsFiber{\tsmaniM}{\tsBaseP}$, we can think of the vector fields $X_i$ as defining a
section of the homomorphism $\tsResFiber{\tsBaseP}: \tsf' \to \tsfdF{\tsBaseP}$. 

Since the vector fields $X_i$ span the tangent space to the fiber and the commutator
of two tangential vector fields is tangential again, we 
have $[X_i, X_j]=\sumConv{\sum_{k=1}^{\tsdimFiber}} {c^k_{ij}} X_k$ for certain functions $c^k_{ij}$.
Since the $X_i$ commute with $\tsg$ it follows that the functions
${c^k_{ij}}$ depend only on the invariants of $\G$.
If we restrict to one of the leaves of the completion of $\G$, then
 the invariants of $\G$ are constant and hence the coefficients $c^k_{ij}$ will be constant.
Locally, the leaves are foliated by the fibers of the projection
and the fact that the coefficients ${c^k_{ij}}$ depend only on the invariants of $\G$
shows that all fibers $\tsFiber{\tsmaniM}{\tsBaseP}$ in the same leaf
of $\vbcompl{\G}$
% of the completion of $\G$
have an isomorphic Lie algebra $\tsfdF{\tsBaseP}$.
So if we move in the direction of the completion of ${\G}$, then the type of 
$\tsfdF{\tsBaseP}$ does not change.
For the same reason the Lie algebras $\tsgdF{\tsBaseP}$ for all
fibers $\tsFiber{\tsmaniM}{\tsBaseP}$ in a leaf of the completion of ${\F}$ have the same type.

The type of $\tsfdF{\tsBaseP}$ is equal to the
type of $\tsgdF{\tsBaseP}$ (reciprocal Lie algebras are anti-isomorphic).
Therefore if we move in the directions of $\F$ and $\G$ the type
of both $\tsgdF{\tsBaseP}$ and $\tsfdF{\tsBaseP}$ does not change.
Hence the type of the reciprocal Lie algebras on the fibers is independent of the choice of
fiber $\tsFiber{\tsmaniM}{\tsBaseP}$.
\end{proof}
\blockpar
We conclude that
the fibers of the projection carry an invariant structure of two reciprocal
Lie algebras. Since the type is locally constant, the type of the Lie
algebra is an \explemph{invariant} of Darboux integrable \hdecomposable systems.

\paragraph{Back to local vector fields}
The next step is to extend the Lie algebras on the fibers to Lie algebras on $\tsmaniM$. The following lemma proves
\theoremref{theorem:main theorem} by constructing
$\tsTLF$ and $\tsTLG$ as subalgebras of $\tsTSf$ and $\tsTSg$, respectively.
\begin{lemma}
\label{lemma:global Lie algebras}
On $\tsmaniM$ there exist
finite-dimensional Lie subalgebras $\tsTLF$, $\tsTLG$ of $\tsTSf$
and $\tsTSg$, respectively, such that for all fibers $\tsFiber{\tsmaniM}{\tsBaseP}$
the restriction map $\tsResFiber{\tsBaseP}$
defines a Lie algebra isomorphism to the Lie algebras on the fibers.
The Lie algebras $\tsTLF$ and $\tsTLG$ are commuting.
\end{lemma}
\begin{proof}
%These reciprocal Lie algebras are closely related to
%$\tsf'$ and $\tsg'$, but these Lie algebras are not finite-dimensional in general.
%%
Choose a basis of vector fields $X_j$ for $\tsTangentV$ contained in $\tsf'$.
The Lie brackets of the $X_j$ define the structure coefficients
\begin{align}
[X_i,X_j] &= \sumConv{\sum_{k=1}^{\tsdimFiber}} c^k_{ij} X_k . 
\label{ts:structure constants}
\end{align}
We already know that these structure coefficients are
functions in $\DarbouxPr_1^*(C^\infty(\tsBaseF))$.
Choose a fiber $\tsFiber{\tsmaniM}{\tsBaseP_0}$.
Since the type of the Lie algebra is constant for each fiber,
there exists for every point $\tsBasePF \in \tsBaseF$ a 
linear transformation $\tsSCtrmap(\tsBasePF) \in \GL(\tsdimFiber,\R)$ such
that the vector fields $Y_i=\tsSCtrmap(\tsBasePF)_i^j X_j$ have the
same structure constants as the restrictions of $X_j$
to the fiber $\tsFiber{\tsmaniM}{\tsBaseP_0}$.
\afootnote{Since the vector fields $X_i$ do not act on the functions in $\DarbouxPr_1^*(C^\infty(\tsBaseF))$!}%
Locally we can arrange that $\tsSCtrmap:\tsBaseF \to \GL(\tsdimFiber,\R)$ is a smooth map,
see \remarkref{remark:smooth lift}.

The new vector fields $Y_j$ have the structure of a
finite-dimensional Lie algebra $\tsTLF$ and this Lie algebra consists
of tangential symmetries of $\G$.
In the same way we can construct
a finite-dimensional Lie subalgebra $\tsTLG$ of $\tsTSg$.
\end{proof}
\blockpar
We call $\tsTLF$ and $\tsTLG$ \defemph{tangential Lie algebras} of
symmetries of $\G$ and $\F$, respectively.
The Lie algebras $\tsTLF$ and $\tsTLG$ are not-unique.
%The choice of a basis for the tangential Lie algebras is not unique.
We can for example multiply a basis $X_j$ of $\tsTLF$ with
a matrix $\glC_j^k \in \GL(\tsdimFiber,\R) \otimes \DarbouxPr_1^*(C^\infty(\tsBaseF))$.
The new vector fields $Y_j= \sumConv{\sum_{l=1}^\tsdimFiber} \glC_j^k X_k$ have
structure coefficients
$d^k_{ij}(x) = \sumConv{\sum_{\alpha,\beta,l=1}^{\tsdimFiber}}\glC^\alpha_i(x) \glC^\beta_j(x) c^l_{\alpha\beta} (\glC^{-1})^k_l(x)$. 
%%
%For all points $x\in \tsmaniM$ the matrix $c_j^k(x)$ acts on the structure constants of $\tsTLF$.
%%
If the new structure coefficients are constants (for example if
 for all points $x$ the matrix ${\glC^k_j(x)}$ is in the stabilizer of the structure constants),
then the vector fields $Y_j$ define a tangential 
Lie algebra of tangential symmetries as well.

\newcommand{\tsStrucSpace}{C}
\newcommand{\tsStrucOrbit}{A}
\newcommand{\tsStrucPoint}{a}

\newcommand{\tsStrucStab}{G_\tsStrucPoint}	% stabilizer group

\begin{acomment}
\newcommand{\Cc}{c}
\begin{remark}
The general linear group acts on the structure coefficients as follows.
\begin{align*}
 Y_k &= \mu^l_k X_l, [Y_i,Y_j]=d^k_{ij} Y_k , \\
d^k_{ij} &= \mu^\alpha_i \mu^\beta_j c^l_{\alpha\beta} (\mu^{-1})^k_l . 
\end{align*}

Elements in the stabilizer of the Lie algebra should satisfy
\begin{align*}
 \Cc^\alpha_i \Cc^\beta_j c^k_{\alpha\beta} =c^\gamma_{ij} \Cc^k_\gamma 
\end{align*}
for all $i,j,k$.
At low dimension ($n=1$ or for abelian Lie algebras) all elements are in the stabilizer (since all structure coefficients are zero). Are there Lie algebras for
which the stabilizer group is trivial?
To make $\tsTLF$ and $\tsTLG$ non-unique we need a continuous non-trivial
path in the stabilizer group.
 \end{remark}
\end{acomment}

\begin{remark}
\label{remark:smooth lift}
Let $\mathcal{L}$ be the space of Lie algebra structures on $\R^\tsdimFiber$.
This space is an algebraic variety in $\tsStrucSpace_\tsdimFiber=\Lambda^2(\R^\tsdimFiber)^* \otimes \R^\tsdimFiber$
defined %by anti-symmetry and 
the Jacobi identity. The group $G=\GL(\tsdimFiber,\R)$ acts 
on $\tsStrucSpace_\tsdimFiber$.
The vector fields $X_j$ from \lemmaref{lemma:global Lie algebras} define a
map $\tsSCmap: \tsBaseF \to \tsStrucSpace_\tsdimFiber$ by assigning to
a point $\tsBasePF \in \tsBaseF$ the structure constants of $X_j$ in
a fiber above $\tsBasePF$ (the structure coefficients are independent of the
point in $\tsBaseG$). Since the vector fields $X_i$ are smooth, the map
$\tsSCmap$ is smooth as well.
By assumption the image of $\tsSCmap$ is contained in a single orbit $\tsStrucOrbit$ of
the action of $G$.

Let $\tsStrucPoint$ be equal to $\tsSCmap(x)$.
Then the orbit $\tsStrucOrbit$ is equal to $G/G_{\tsStrucPoint}$, where $G_{\tsStrucPoint}$
 is the stabilizer subgroup of the point $\tsStrucPoint$ in the orbit.
We want to prove that there is a smooth lift of the map $\tsSCmap$ to
a map $\tsSCtrmap: \tsBaseF \to \GL(\tsdimFiber,\R)$ such that the diagram below
is commutative.
%%%%%%%%%%%%%%%%%%%%%%%%%%%%%%%%%%%%%%%%%%%%%%%%%%%%
% \begin{xydiagram}
% \xymatrix{
%  & \GL(\tsdimFiber,\R) \ar[d]  \\
% \tsBaseF \ar[r]^{\tsSCmap} \ar@{-->}[ru]^{\tsSCtrmap} & \tsStrucOrbit \\
% }
% \end{xydiagram}
%%%%%%%%%%%%%%%%%%%%%%%%%%%%%%%%%%%%%%%%%%%%%%%%%%%%
\begin{xydiagram}
\xymatrix{
 & \GL(\tsdimFiber,\R) \ar[d]  \\
\tsBaseF \ar[r]^{\tsSCmap} \ar@{-->}[ru]^{\tsSCtrmap}
	 & \tsStrucOrbit \save[]+<.7cm,0cm>*{\subset \tsStrucSpace_\tsdimFiber} \restore
		 \save[]+<-.05cm,-.45cm>*{\begin{turn}{-90}=\end{turn}} \restore 
		 \save[]+<.0cm,-0.9cm>*{G/G_a} \restore 
%%%
%		 \save[]+<.5cm,-.5cm>*{\begin{rotate}{-90}{$=G/G_a$}\end{rotate}} \restore\\
}
\end{xydiagram}
%%%%%%%%%%%%%%%%%%%%%%%%%%%%%%%%%%%%%%%%%%%%%%%%%%%%
%%
We have to be carefull here because the map $\tsSCmap$ is continuous
with respect to the topology on $\tsStrucOrbit$ induced from
the surrounding space $\tsStrucSpace_{\tsdimFiber}$.
The structure on $\tsStrucOrbit$ as the homogeneous space $G/\tsStrucStab$
might be different.
From the theory in~\cite[\mathpara A.3]{Eendebak2006} it follows
that $\tsSCmap$ is a smooth map to $G/\tsStrucStab$.
The projection $G\to G/\tsStrucStab$ is a principal fiber bundle and hence there is a smooth lift
of $\tsSCmap$.
\end{remark}

%\begin{acomment}
Using the Lie algebras $\tsTLF$ and $\tsTLG$ we can construct
a normal form for any Darboux integrable system.
\begin{theorem}
\label{theorem:Darboux integrable system}
Let $(\tsmaniM,\F,\G)$ be a Darboux integrable system
satisfying~\eqref{decomposable EDS:assumptions}.
 %Let $\niF$ and $\niG$ be the rank of $\F$ and $\G$, respectively.
 Then locally there is a unique local Lie group $H$ such that the
manifold $M$ is of the form
\begin{align*}
\tsBaseF \times \tsBaseG \times H ,
\end{align*}
with $\tsBaseF \subset \R^{\niG}$, $\tsBaseG \subset \R^{\niF}$. 
The Darboux projection $\DarbouxPr$ is given by the 
projection on $\tsBaseF \times \tsBaseG$.
The tangential symmetries of $\F$ and $\G$ are tangent to the fibers of $\pi$ and restricted
to each fiber they form reciprocal Lie algebras.
The left- and \rightinvariant vector fields on $H$ define the
tangential Lie algebras of symmetries on $M$.
\end{theorem}
%\end{acomment}

%%%

\paragraph{Commuting part}
In each fiber $\tsFiber{\tsmaniM}{\tsBaseP}$ of the projection, we can consider the centers of the Lie algebras
on the fiber. The center of $\tsfdF{\tsBaseP}$ is equal
to the center of  $\tsgdF{\tsBaseP}$.
We write $\tsCenterF{\tsBaseP}$ for the center of the Lie algebra on $\tsFiber{\tsmaniM}{\tsBaseP}$.
The distribution $\tsCenterDistr$ spanned by $\tsCenterF{\tsBaseP}$, $\tsBaseP \in \tsBase$, is invariantly defined and is
integrable and hence defines a local foliation of the fibers.

\begin{lemma}
The tangential vector fields that are symmetries of both $\F$ and $\G$ are tangential vector fields for which the restriction to a fiber is in the center of the Lie algebra of tangential symmetries. In particular, these vector fields are contained
in $\tsCenterDistr$.
The tangential symmetries of a \hdecomposable system are equal
to the center of $\tsTLF$ (which is equal to the center of $\tsTLG$).

The space of tangential symmetries of the system is equal
to the center of $\tsTLF$ (which is equal to the center of $\tsTLG$).
\end{lemma}
\begin{proof}
Restricted to each fiber $\tsFiber{\tsmaniM}{\tsBaseP}$ a tangential symmetry is invariant under both
$\tsfdF{\tsBaseP}$ and $\tsgdF{\tsBaseP}$. 
Since the restriction is invariant under $\tsgdF{\tsBaseP}$ it must
be contained in $\tsfdF{\tsBaseP}$. Since the restriction 
commutes with $\tsfdF{\tsBaseP}$ it is by definition contained in the
center of $\tsfdF{\tsBaseP}$.
%By \theoremref{theorem:center is left- and right-invariant} the restriction
%must be contained in $\tsCenterF{\tsBaseP}$.

From the definitions if is clear that any element from the center
of $\tsTLF$ is a tangential symmetry. Using the fact that $\tsTLF$ spans $\tsCenterDistr$ we can show that the center of $\tsTLF$ contains
all tangential symmetries.
\end{proof}
\blockpar
The converse is not true. Not every vector field contained in $\tsCenterDistr$
is a tangential symmetry of $\F$ or $\G$.
This lemma shows that, unless the Lie group associated to a \hdecomposable system is abelian, the method of Darboux does not coincide with the symmetry methods
for partial differential equations introduced by Lie.

\begin{remark}
Given a Lie group there is no systematic way of constructing
a corresponding Darboux integrable
system. Even if we can construct such a system, there is no guarantee that this system corresponds
to a prolonged first order system or second order equation.
Despite this, Vessiot~\cite{Vessiot1939,Vessiot1942b}
succeeded in using the classification of 3-dimensional Lie algebras to
construct a classification of all Darboux integrable
 \hdecomposable Goursat equations.
\end{remark}

%%%%%%%%%%%%%%%%%%%%%%%%%%%%%%%%%%%%%%%%%%%%%%%%%%%%%%%%%%%%%%%%%%%%%%%%%
%%%%%%%%%%%%%%%%%%%%%%%%%%%%%%%%%%%%%%%%%%%%%%%%%%%%%%%%%%%%%%%%%%%%%%%%%

\subsection{Examples}

%%%
\begin{example}[Wave equation]
On $\tsmaniM=\R^5$ with coordinates $x,y,z,p,q$ introduce the following two distributions:
\begin{align*}
\F &= \vbspand{ \p_x + p\p_z, \p_p } , \quad
	\G = \vbspand{ \p_y +q \p_z, \p_q} .
\end{align*}
The triple $(\tsmaniM, \F, \G)$ defines a Darboux integrable \hdecomposable system.
The contact structure on the first order equation manifold of the wave equation $\pfracTwoflat{z}{x}{y}=0$
is given precisely by this structure.

The invariants of $\F$ are $y,q$ and the invariants of $\G$ are $x,p$. The Darboux projection 
is given by
\begin{align*}
\DarbouxPr: \tsmaniM \to \R^2 \times \R^2 : (x,y,z,p,q) \mapsto (x,p) \times (y,q) .
\end{align*}
 As a set of commuting vector fields on $\R^2\times\R^2$ we take
$\ti{F}_1=\p_x$, $\ti{F}_2=\p_p$,
$\ti{G}_1=\p_y$, $\ti{G}_2=\p_y$. The lifts to vector fields on $\tsmaniM$ are given by
\begin{align*}
F_1 &= \p_x + p\p_z, \quad F_2 = \p_p, \\
G_1 &= \p_y +q\p_z, \quad G_2=\p_q.
\end{align*}
Define $F_3=[F_1,F_2]=-\p_z$ and $G_3=[G_1,G_2]=-\p_z$. The fibers of the projection $\DarbouxPr$ are isomorphic to $\R$ are we can use
$z$ as a coordinate on the fibers.
On the leaves of the completion of ${\G}$ (\ie $x$ and $p$ constant) we have a 3-dimensional Lie algebra with
structure equations
\begin{align*}
\begin{split}
[F_1,F_2] &=F_3, \eqspace
[F_1, F_3] =0, \eqspace
[F_2, F_3] =0 .
\end{split}
\end{align*}
The same holds for the leaves of the completion of ${\F}$. The fibers of the projection are 1-dimensional and have the structure of a 1-dimensional abelian Lie algebra (generated by the vector field $\p_z$).
\end{example}

\newcommand{\Ck}{k}	% constant k
\newcommand{\CH}{H}	% constant k

\begin{example}
Consider the \hdecomposable second order equation
\begin{align}
z_{xy} &= \frac{a z}{(x+y)^2} .
\end{align}
This equation is Darboux integrable on the $\Ck+1$-jets if $a=\Ck(\Ck+1)$. Hence
for $a=\Ck(\Ck+1)$ the $(\Ck-1)$-th prolonged equation manifold has dimension $5+2\Ck$ and the
prolonged Monge systems define a Darboux integrable \hdecomposable system.

\newcommand{\zxd}{z_{xxx}}
\newcommand{\zyd}{z_{yyy}}

We work out the case $\Ck=2$ in detail. The prolonged equation manifold
has coordinates $x$, $y$, $z$, $p=z_x$, $q=z_y$, $r=z_{xx}$, $t=z_{yy}$, $\zxd$, $\zyd$. The prolonged characteristic systems
are $\F=\vbspant{F_1,F_2}$, $\G=\vbspant{G_1,G_2}$ with 
\begin{align*}
F_1 &= \p_x + p\p_z + r\p_p + \sigma \p_q + \zxd \p_r + Y(\sigma) \p_t
	+ Y(Y(\sigma)) \p_{\zyd}, \\
F_2 &= \p_{\zxd}, \\
G_1 &= \p_y + q\p_z + \sigma\p_p + t\p_q +X(\sigma) \p_r + \zyd \p_t  + X(X(\sigma)) \p_{\zxd} ,\\
G_2 &= \p_{\zyd} .
\end{align*}
Here $\sigma=6z/(x+y)^2$, $X=\p_x+p\p_z+r\p_p + \sigma \p_q$ and
$Y=\p_y+q\p_z+\sigma\p_p + t \p_q$.
Both characteristic systems have two invariants,
%%%%%
\begin{align*}
 I_\F &=\verz{y,\zyd+6\frac{q+t(x+y)}{(x+y)^2} }, \\
 I_\G &=\verz{x,\zxd+6\frac{p+r(x+y)}{(x+y)^2} } .
\end{align*}
%%%%%
If we make a transformation to the variables 
$\trans{x}=x$, $\trans{y}=y$, $\trans{z}=z$, 
$\trans{p}=p$, $\trans{q}=q$,
$\trans{r}=r$, $\trans{t}=t$, $\trans{a}=\zxd+6\fracflat{(p+r(x+y))}{(x+y)^2}$,
$\trans{b}=\zyd+6\fracflat{(q+t(x+y))}{(x+y)^2}$, then
we can easily construct the Darboux projection, choose commuting vector fields,
lifts these vector fields to $\tsmaniM$ and calculate the tangential symmetries.
\begin{acomment}
In the new variables (omitting the tilde)
\begin{align*}
L_1 &= \p_t, \\
L_2 &= \p_q-6H \p_t , \\
L_3 &= \p_z-6H\p_q+24H^2\p_b , \\
L_4 &= H\p_z - H^2\p_p-3H^2\p_q +2H^3\p_r + 10H^3\p_t  , \\
L_5 &= H^2\p_z-2H^3 (\p_p +\p_q)+6H^4(\p_r+\p_t) .
\end{align*}
\end{acomment}
Translated back to the original variables we find that the tangential symmetries of
$\F$ are spanned by
\begin{align*}
L_1 &= \p_t-6\CH\p_{\zyd}, \\
L_2 &= \p_q-6\CH \p_t +30 \CH^2\p_{\zyd}, \\
L_3 &= \p_z-6\CH\p_q+24 \CH^2\p_b -108 \CH^3 \p_{\zyd}, \\
L_4 &= \CH\p_z - \CH^2\p_p-3\CH^2\p_q +2\CH^3\p_r \\
	&\eskip\ + 10\CH^3\p_t +6\CH^4\p_{\zxd} +42\CH^4\p_{\zyd} , \\
L_5 &= \CH^2\p_z-2\CH^3 (\p_p +\p_q) \\
	&\eskip\ +6\CH^4(\p_r+\p_t) -24\CH^4(\p_{\zxd}+\p_{\zyd}) ,
\end{align*}
with $\CH=(x+y)^{-1}$.
The tangential symmetries of $\G$ are given by the same expressions, but
with $x$, $p$, $r$ and $\zxd$ replaced by $y$, $q$, $t$ and $\zyd$, respectively.
The tangential symmetries
commute and the Lie group associated to this Darboux integrable system is the
5-dimensional abelian Lie group.
\begin{acomment}

Solving the equation by the method of Darboux leads to the general solution
\begin{align*}
 z(x,y) &= \pm (x+y)^{k+1} \frac{\p^{2k}}{\p x^k \p y^k} \left( \frac{X(x)+Y(y)}{x+y} \right) .
\end{align*}
% See~\cite[\cpage ?]{Bryant1995-I} and \cite[Chapter 3]{Darboux1894I}.
\end{acomment}
\end{example}

	% main article

%% References
\appendix
\refstepcounter{section}
\addcontentsline{toc}{section}{Bibliography}
\small
\bibliography{references}

\end{document}